\documentclass[a4paper,reqno,11pt]{amsart}
\usepackage{amssymb}
\usepackage[mathcal]{euscript}
\usepackage[cmtip,all]{xy}

\newcommand{\bydef}{:=}

\newcommand{\lspan}[1]{\mathrm{span}\left\{#1\right\}}

\newcommand{\tr}{\mathrm{tr}}
\newcommand{\sr}{\mathrm{sr}}

\newcommand{\cC}{\mathcal{C}}

\newcommand{\cK}{\mathcal{K}}

\newcommand{\cN}{\mathcal{N}}
\newcommand{\cO}{\mathcal{O}}
\newcommand{\cQ}{\mathcal{Q}}

\newcommand{\cS}{\mathcal{S}}

\newcommand{\cU}{\mathcal{U}}
\newcommand{\cV}{\mathcal{V}}
\newcommand{\cW}{\mathcal{W}}

\newcommand{\ZZ}{\mathbb{Z}}

\newcommand{\CC}{\mathbb{C}}

\newcommand{\FF}{\mathbb{F}}

\newcommand{\chr}[1]{\mathrm{char}\,#1}

\DeclareMathOperator{\Aut}{\mathrm{Aut}}
\DeclareMathOperator{\AAut}{\mathbf{Aut}}
\DeclareMathOperator{\Der}{\mathrm{Der}}

\DeclareMathOperator{\Centr}{\mathrm{Centr}}

\newcommand{\Lie}{\mathrm{Lie}}
\newcommand{\ad}{\mathrm{ad}}

\newcommand{\frsl}{{\mathfrak{sl}}}

\newcommand{\Gs}{\mathbf{G}}
\newcommand{\Hs}{\mathbf{H}}
\newcommand{\Ks}{\mathbf{K}}
\newcommand{\Ds}{\mathbf{D}}

\newcommand{\bmu}{\boldsymbol{\mu}}

\def\bigstrut{\vrule height 16pt width 0ptdepth 6pt}

\newenvironment{romanenumerate}
 {\begin{enumerate}
 
 }{\end{enumerate}}

\newtheorem{theorem}{Theorem}
\newtheorem{proposition}[theorem]{Proposition}
\newtheorem{lemma}[theorem]{Lemma}

\theoremstyle{definition}
\newtheorem{df}[theorem]{Definition}

\theoremstyle{remark}
\newtheorem{remark}[theorem]{Remark}

\numberwithin{equation}{section}

\begin{document}

\title{Okubo algebras: automorphisms, derivations and idempotents}

\author{Alberto Elduque}
\address{Departamento de Matem\'aticas
 e Instituto Universitario de Matem\'aticas y Aplicaciones,
 Universidad de Zaragoza, 50009 Zaragoza, Spain}
\email{elduque@unizar.es}
\thanks{Supported by the Spanish Ministerio de Econom\'{\i}a y Competitividad---Fondo Europeo de Desarrollo Regional (FEDER) MTM2010-18370-C04-02 and of the Diputaci\'on General de Arag\'on---Fondo Social Europeo (Grupo de Investigaci\'on de \'Algebra)}



\date{}

\begin{abstract}
A survey of some properties of Okubo algebras is presented here. Emphasis is put on automorphisms and derivations of these algebras, especially in characteristic three, where the situation is more involved and interesting.
In this case, the Okubo algebra is closely related to a nodal noncommutative Jordan algebra. These latter algebras have appeared several times in the work of Helmut Strade.
\end{abstract}

\dedicatory{Dedicated to Helmut Strade on the occasion of his $70^\text{th}$ birthday}

\maketitle

\bigskip

\section{Introduction}
One of the first papers by Helmut Strade \cite{Str72} is about nodal noncommutative Jordan algebras and their relationship to simple modular Lie algebras. I was aware of this paper when dealing with Okubo algebras, which form an important class of symmetric composition algebras (see \cite[Ch.~VIII]{KMRT}), and are quite useful in dealing with the phenomenon of triality in dimension $8$. It turns out that the behavior, and even the definition, of these algebras is quite different over fields of characteristic three. The computation of its group of automorphisms (rational points) and of its Lie algebra of derivations in this case was performed in \cite{Eld99} by relating it to some nodal noncommutative Jordan algebras.
The aim of this paper is to survey some of these results.

In order to avoid technical details, we will work over an algebraically closed field $\FF$. The \emph{Okubo algebra} $\cO$ over $\FF$ will be defined in the next section, where a particular grading by $\ZZ_3^2$ will be highlighted. The relations of the Okubo algebra with a nodal noncommutative Jordan algebra in characteristic three will be given in Section 3, and these will be used to compute the Lie algebra of derivations of $\cO$. Section 4 will be devoted to the automorphisms of $\cO$. It turns out that in characteristic three, the subalgebra of elements fixed by all the automorphisms of $\cO$ is spanned by a distinguished idempotent. Idempotents are important for the study of the Okubo algebras, and more generally of the symmetric composition algebras. Section 5 will be devoted to them.

Most of the mathematical career of Helmut Strade has been devoted to the classification of the finite-dimensional simple modular Lie algebras \cite{Str1,Str2,Str3} over an algebraically closed field. Very soon one learns that the behavior of these algebras is much worse than in characteristic zero. Here, in characteristic three, the Lie algebra $\Der(\cO)$ will be shown to present some of the difficulties that appear in prime characteristic:
\begin{itemize}
\item $\Der(\cO)$ is a semisimple Lie algebra of dimension $8$, but not a direct sum of simple Lie algebras.
\item $[\Der(\cO),\Der(\cO)]$ is simple (of Hamiltonian type), but its Killing form is trivial, and it has outer derivations.
\item The affine group scheme of automorphism $\AAut(\cO)$ is not smooth. Its dimension is $8$. The two missing dimensions ($\dim\Der(\cO)=10$) are explained in terms of the $\ZZ_3^2$-grading mentioned above.
\end{itemize}

\smallskip

Throughout the paper, $\FF$ will denote an algebraically closed field.

\bigskip

\section{Okubo algebras}

Assume first that the characteristic of $\FF$ is not three, so $\FF$ contains a primitive cubic root $\omega$ of $1$. Let $M_3(\FF)$ be the associative algebra of $3\times 3$-matrices over $\FF$. Let $\frsl_3(\FF)$ denote the corresponding special Lie algebra. Define a binary multiplication on $\frsl_3(\FF)$ as follows:
\begin{equation}\label{eq:x*y}
x*y=\omega xy -\omega^2yx-\frac{\omega-\omega^2}{3}\tr(xy)1.
\end{equation}
(Here $\tr$ denotes the usual trace and the product of matrices is denoted by juxtaposition. Note that $\tr(x*y)=0$, so this multiplication on $\frsl_3(\FF)$ is well defined.)

A straightforward computation \cite{EM93} shows that
\begin{equation}\label{eq:x*y*x}
(x*y)*x=x*(y*x)=\sr(x)y
\end{equation}
for any $x,y\in \frsl_3(\FF)$, where $\sr$ denotes the quadratic form that appears as the coefficient of $X$ in the Cayley--Hamilton polynomial
\[
\det(X1-a)=X^3-\tr(a)X^2+\sr(a)X-\det(a)
\]
for $a\in M_3(\FF)$. Note that the restriction of $\sr$ to $\frsl_3(\FF)$ is nonsingular. Also, if $\chr\FF\ne 2$, $\sr(a)=\frac{1}{2}\left(\tr(a)^2-\tr(a^2)\right)$ for any $a$. The polar form $\sr(a,b)\bydef \sr(a+b)-\sr(a)-\sr(b)$ satisfies
\[
\sr(a,b)=\tr(a)\tr(b)-\tr(ab)
\]
for any $a,b\in M_3(\FF)$, and this is valid in any characteristic.

As in \cite[p.~xix]{KMRT}, a quadratic form $q:V\rightarrow \FF$ on a finite dimensional vector space is said to be \emph{nonsingular} if either its polar form ($q(v,w)\bydef q(v+w)-q(v)-q(w)$) is nondegenerate (i.e., $V^\perp=0$) or $\dim V^\perp =1$ and $q(V^\perp)\ne 0$. The latter case occurs only if $\chr\FF=2$ and $V$ is odd-dimensional.

\begin{lemma}[{see \cite{OO81} or \cite[Lemma 4.42]{EKmon}}]\label{le:symmetric}
Let $(\cS,*,n)$ be an algebra endowed with a nonsingular quadratic form $n$. Then $n$ is multiplicative and its polar form is associative if and only if it satisfies
\begin{equation}\label{eq:xyx}
(x*y)*x=n(x)y=x*(y*x)
\end{equation}
for any $x,y\in\cS$.
\end{lemma}

Therefore, the vector space $\cO=\frsl_3(\FF)$ endowed with the multiplication $*$ in \eqref{eq:x*y} and the nonsingular quadratic form $n(x)=\sr(x)$ is a \emph{symmetric composition algebra}:
\begin{equation}\label{eq:symmetric}
n(x*y)=n(x)n(y),\quad n(x*y,z)=n(x,y*z),
\end{equation}
for any $x,y,z$.

The algebra $(\cO,*)$ is \emph{Lie-admissible} in the sense that the algebra
\begin{equation}\label{eq:O*-}
(\cO,*)^-\bydef\bigl(\cO,[u,v]^*\bydef u*v-v*u\bigr)
\end{equation}
is a Lie algebra. In fact, the map $x\mapsto -x$ gives an isomorphism from $\frsl_3(\FF)$ onto the algebra $(\cO,*)^-$.

\begin{df}\label{df:Okubo_not3}
The symmetric composition algebra $(\cO,*,n)$ is called the \emph{Okubo algebra} over $\FF$ ($\chr\FF\ne 3$).
\end{df}

It turns out that this algebra already appeared, in a completely different guise, in \cite{Petersson}.

\medskip

In order to define the Okubo algebra over fields of characteristic three consider, as in \cite{Twisted}, the Pauli matrices:
\[
x=\begin{pmatrix} 1&0&0\\ 0&\omega&0\\ 0&0&\omega^2 \end{pmatrix},
\qquad
y=\begin{pmatrix} 0&0&1\\ 1&0&0\\ 0&1&0\end{pmatrix},
\]
in $M_3(\CC)$, which satisfy
\begin{equation}\label{eq:xyomega}
x^3=y^3=1,\quad xy=\omega yx,
\end{equation}
and consider the following elements of $M_3(\CC)$:
\[
x_{i,j}\bydef \frac{\omega^{ij}}{\omega-\omega^2}x^iy^j,
\]
for $i,j\in\ZZ$. Then $x_{i,j}$ only depends on the classes of $i$ and $j$ modulo $3$. The elements $x_{i,j}$ for $-1\leq i,j\leq 1$ constitute a basis of $M_3(\CC)$, and the elements $x_{i,j}$ for $-1\leq i,j\leq 1$, $(i,j)\ne (0,0)$, constitute a basis of $\frsl_3(\CC)$.

For $a,b\in M_3(\CC)$, write $a\diamond b=\omega ab-\omega^2ba$. Then, for $i,j,k,l\in\ZZ$ we get, using \eqref{eq:xyomega}:
\[
\begin{split}
x_{i,j}\diamond x_{k,l}&=\frac{\omega^{1-\Delta}-\omega^{\Delta -1}}{\omega-\omega^2}x_{i+k,j+l}\\
    &=\begin{cases} x_{i+k,j+l}&\text{if $\Delta\equiv 0\pmod 3$,}\\
            0&\text{if $\Delta\equiv 1\pmod 3$,}\\
            -x_{i+k,j+l}&\text{if $\Delta\equiv 2\pmod 3$,}
    \end{cases}
\end{split}
\]
where $\Delta=\left|\begin{smallmatrix} i&j\\ k& l\end{smallmatrix}\right| =il-jk$. Miraculously, the $\omega$'s disappear!

Besides, for $u,v\in\frsl_3(\CC)$,
\[
\begin{split}
u*v&=u\diamond v-\frac{1}{3}\tr(u\diamond v)1\\
    &=u\diamond v+\tr(uv)x_{0,0}\\
    &=u\diamond v-\sr(u,v)x_{0,0}\\
    &=u\diamond v-n(u,v)x_{0,0},
\end{split}
\]
or
\begin{equation}\label{eq:xdiamondy}
u\diamond v=n(u,v)x_{0,0}+u*v.
\end{equation}
Thus, for $-1\leq i,j,k,l\leq 1$ and $(i,j)\ne (0,0)\ne (k,l)$,
\begin{equation}\label{eq:nxijxkl}
\begin{cases}
n(x_{i,j},x_{k,l})=1&\text{for $(i,j)=-(k,l)$,}\\
n(x_{i,j},x_{k,l})=0&\text{otherwise.}
\end{cases}
\end{equation}
Also,
\begin{equation}\label{eq:nxij}
n(x_{i,j})=0
\end{equation}
for any such $(i,j)$.

\smallskip

The $\ZZ$-span
\[
\cO_\ZZ=\text{$\ZZ$-}\lspan{x_{i,j}:-1\leq i,j\leq 1,\ (i,j)\ne (0,0)}
\]
is then closed under $*$, and $n(\cO_\ZZ)\subseteq \ZZ$. The $\ZZ$-algebra $(\cO_\ZZ,*,n)$ satisfies equations \eqref{eq:symmetric}, as it is a $\ZZ$-subalgebra of $(\cO,*,n)$.

Hence, for any field $\FF$, \emph{even in characteristic three}, we may consider the vector space
\[
\cO_\FF\bydef\cO_\ZZ\otimes_\ZZ\FF,
\]
which is endowed with the natural extension of the multiplication $*$ and the norm $n$ in $\cO_\ZZ$. By equations \eqref{eq:nxijxkl} and \eqref{eq:nxij}, the norm is nonsingular (actually hyperbolic) and the resulting algebra satisfies equations \eqref{eq:symmetric}, so it is a symmetric composition algebra over $\FF$.

If the characteristic is not three, then this algebra is (isomorphic to) the Okubo algebra in Definition \ref{df:Okubo_not3}. Therefore, we will write $\cO=\cO_\FF$ for any $\FF$ and
we may then extend Definition \ref{df:Okubo_not3} as follows:

\begin{df}\label{df:Okubo}
The symmetric composition algebra $(\cO,*,n)$ is called the \emph{Okubo algebra} over $\FF$.
\end{df}

In our basis $\{x_{i,j}:-1\leq i,j\leq 1,\ (i,j)\ne (0,0)\}$, the multiplication table is given in Table \ref{tb:Okubo}. This is independent of the characteristic.

\begin{table}[h!]
\begin{tabular}{c|c@{}c|c@{}c|c@{}c|c@{}c|}
    $*$&$x_{1,0}$&$x_{-1,0}$&$x_{0,1}$&$x_{0,-1}$
    &$x_{1,1}$&$x_{-1,-1}$&$x_{-1,1}$&$x_{1,-1}$\\
    \hline
\bigstrut
   $x_{1,0}$&$x_{-1,0}$&$0$&$0$&$-x_{1,-1}$&
   $0$&$-x_{0,-1}$&$0$&$-x_{-1,-1}$\\
\bigstrut
   $x_{-1,0}$&$0$&$x_{1,0}$&
   $-x_{-1,1}$&$0$&
   $-x_{0,1}$&$0$&
			$-x_{1,1}$&$0$\\
\hline
\bigstrut
   $x_{0,1}$&$-x_{1,1}$&$0$&
   $x_{0,-1}$&$0$&
   $-x_{1,-1}$&$0$&
			$0$&$-x_{1,0}$\\
\bigstrut
   $x_{0,-1}$&$0$&$-x_{-1,-1}$&
   $0$&$x_{0,1}$&
   $0$&$-x_{-1,1}$&
			$-x_{-1,0}$&$0$\\
\hline
\bigstrut
   $x_{1,1}$&$-x_{-1,1}$&$0$&
   $0$&$-x_{1,0}$&
   $x_{-1,-1}$&$0$&
			$-x_{0,-1}$&$0$\\
\bigstrut
   $x_{-1,-1}$&$0$&$-x_{1,-1}$&
   $-x_{-1,0}$&$0$&
   $0$&$x_{1,1}$&
			$0$&$-x_{0,1}$\\
\hline
\bigstrut
   $x_{-1,1}$&$-x_{0,1}$&$0$&
   $-x_{-1,-1}$&$0$&
   $0$&$-x_{1,0}$&
			$x_{1,-1}$&$0$\\
\bigstrut
   $x_{1,-1}$&$0$&$-x_{0,-1}$&
   $0$&$-x_{1,1}$&
   $-x_{-1,0}$&$0$&
			$0$&$x_{-1,1}$\\
\hline
\end{tabular}
\caption{\bigstrut Multiplication table of the Okubo algebra}\label{tb:Okubo}
\end{table}

The assignment $\deg(x_{i,j})=(i,j)\pmod 3$ gives a grading on $\cO$ by $\ZZ_3^2$. This grading will play a key role later on. 

Using this grading it becomes clear that the commutative center
\begin{equation}\label{eq:KO}
\cK(\cO,*)\bydef \{u\in\cO: u*v=v*u\ \forall v\in\cO\}
\end{equation} 
is trivial. Hence $(\cO,*,n)$ is not the para-Cayley algebra. (See Remark \ref{re:para-Hurwitz} later on.)

\begin{remark}
Over arbitrary fields (not necessarily algebraically closed), the symmetric composition algebra with multiplication table as in Table \ref{tb:Okubo} is called the \emph{split Okubo algebra} over $\FF$. An \emph{Okubo algebra} is, by definition, a twisted form of this algebra (i.e., it becomes isomorphic to the split Okubo algebra after an extension of scalars).

The Okubo algebra was defined in \cite{Oku78} (see also \cite{Okubo_book}) using a slight modification of \eqref{eq:x*y} over fields of characteristic not three, and then in a different way in characteristic three. It was termed the \emph{pseudo-octonion algebra}.

The classification of the Okubo algebras is given in \cite{EM91,EM93} in characteristic not three, and in \cite{Eld97} in characteristic three (see also \cite[\S 12]{CEKT}).
\end{remark}

\begin{remark}
The construction above of the split Okubo algebra over an arbitrary field is reminiscent of the use of Chevalley bases to define the classical simple Lie algebras over arbitrary fields. (See \cite[\S 4.1]{Str1}.)
\end{remark}

\bigskip

\section{Derivations}

Equation \eqref{eq:x*y*x} shows that any derivation $d$ of the Okubo algebra $(\cO,*)$ belongs to the orthogonal Lie algebra relative to its norm $n$.

If the characteristic of $\FF$ is not three, then $n(x,y)=-\tr(xy)$ for any $x,y\in \cO=\frsl_3(\FF)$. Then \eqref{eq:x*y} easily gives that, if we extend $d$ to $M_3(\FF)$ by $d(1)=0$, we obtain a derivation of $M_3(\FF)$, and conversely. As any derivation of $M_3(\FF)$ is of the form $\ad_x: y\mapsto [x,y]=xy-yx$ for an element $x\in \frsl_3(\FF)$, the next result follows.

\begin{theorem}
Let $\FF$ be a field of characteristic not three, then $\Der(\cO,*)$ is isomorphic to $\frsl_3(\FF)$. Actually,
\[
\Der(\cO,*)=\{ \ad_x^*\,(:y\mapsto x*y-y*x): x\in\cO\}.
\]
\end{theorem}

Note that this means that for any $u\in\cO$, the map $\ad_u^*:v\mapsto [u,v]^*\bydef u*v-v*u$, is a derivation of $u$, and these derivations span $\Der(\cO,*)$. Hence, in particular, for any $u\in\cO_\ZZ$ $\ad_u^*$ is a derivation of $(\cO_\ZZ,*)$.

Now assume that the characteristic of $\FF$ is three. Since $(\cO,*,n)$ is obtained by an extension of scalars: $\cO=\cO_\ZZ\otimes_\ZZ \FF$, we have that the subspace spanned by the adjoint maps $\ad_u^*$ is a subspace (actually an ideal) of the Lie algebra of derivations. However, in characteristic three, this is not the whole $\Der(\cO,*)$.

\smallskip

Let $\FF[x,y]$ be the truncated polynomial algebra $\FF[X,Y]/(X^3-1,Y^3-1)$, with $x$ and $y$ the classes of $X$ and $Y$ respectively (so $x^3=1=y^3$), and let $\FF_0[x,y]=\lspan{x^iy^j: 0\leq i,j\leq 2,\ (i,j)\ne (0,0)}$. Consider the new multiplication on $\FF[x,y]$ given by
\begin{equation}\label{eq:diamond_product}
x^iy^j\diamond x^ky^l=(1-\Delta)x^{i+k}y^{j+l},
\end{equation}
with $\Delta$ being the determinant $\left|\begin{smallmatrix} i&j\\ k&l\end{smallmatrix}\right|$. Then, for $u,v\in\FF_0[x,y]$,
\[
u\diamond v=n(u,v)1+u*v,
\]
with $n(u,v)\in\FF$ and $u*v\in\FF_0[x,y]$ (as $\FF[x,y]=\FF 1\oplus\FF_0[x,y]$).

The arguments in the previous section prove the next result:

\begin{theorem}[\cite{Eld99}]\label{th:Okubo_nodal}
The Okubo algebra $(\cO,*,n)$ ($\chr\FF=3$)  is isomorphic to the algebra $(\FF_0[x,y],*,n)$ defined above.
\end{theorem}

The following result can be checked easily.

\begin{proposition}[\cite{Eld99}]\label{pr:Okubo_nodal}
The product $\diamond$ in \eqref{eq:diamond_product}, defined in $\FF[x,y]$, is given by the formula
\[
f\diamond g=fg-\left(\frac{\partial f}{\partial x}\frac{\partial g}{\partial y}-\frac{\partial g}{\partial x}\frac{\partial f}{\partial y}\right)xy.
\]
\end{proposition}

The algebra $(\FF[x,y],\diamond)$ ($\chr\FF=3$) is a simple nodal noncommutative Jordan algebra (see \cite{Kokoris58}). The Lie algebra of derivations of $(\cO,*)$ is determined in terms of the derivations of this nodal noncommutative Jordan algebra.

\begin{theorem}[{\cite[Theorem 4]{Eld99}}]\label{th:derivations}
The Lie algebra of derivations $\Der(\cO,*)$ ($\chr\FF=3$) is isomorphic to the Lie algebra of derivations of the simple nodal noncommutative Jordan algebra $(\FF[x,y],*)$. The isomorphism assigns to any derivation $d$ of $(\FF[x,y],*)$ its restriction to $\FF_0[x,y]$, and this is identified with $\cO$ by means of Theorem \ref{th:Okubo_nodal}.

Moreover, $\Der(\cO,*)$ is a ten-dimensional semisimple Lie algebra. Its derived Lie algebra $[\Der(\cO,*),\Der(\cO,*)]$ is simple and coincides with the span of the inner derivations $\ad_u^*$, for $u\in \cO$.
\end{theorem}

\begin{remark}
The $\ZZ_3^2$-grading on $(\cO,*,n)$ induces a $\ZZ_3^2$-grading on $\Der(\cO,*)$. For any $(0,0)\ne (i,j)\in\ZZ_3^2$, $\Der(\cO,*)_{(i,j)}=\ad_{\cO_{(i,j)}}^*$ has dimension one, while $\Der(\cO,*)_{(0,0)}$ is a toral two-dimensional algebra spanned by $x\frac{\partial\ }{\partial x}$ and $y\frac{\partial\ }{\partial y}$ (where we identify, as before, $\cO$ with $\FF_0[x,y]$). For any $u\in\cO_{(i,j)}$ , $(i,j)\ne (0,0)$, $\bigl(\ad_u^*\bigr)^3$ is in $\Der(\cO,*)_{(0,0)}$.
\end{remark}

\begin{remark}
The simple Lie algebra $[\Der(\cO,*),\Der(\cO,*)]$ ($\chr\FF=3$) is then isomorphic to the Lie algebra $(\cO,*)^-$ in \eqref{eq:O*-}, or to $(\FF_0[x,y],[.,.]^*)$. The Lie bracket here is determined by
\[
[x^iy^j,x^ky^l]^*=\begin{vmatrix} i&j\\ k&l\end{vmatrix}x^{i+k}y^{j+l},
\]
for $0\leq i,j\leq 2$, $(i,j)\ne (0,0)\ne (k,l)$. This is an instance of a \emph{Block algebra} \cite{Block58}, and it was also considered by Albert and Frank \cite{AF}.

This eight-dimensional simple Lie algebra is of Cartan type 
$H(2,(1,1),\omega)^{(2)}$ for a suitable $\omega$ (see \cite[Lemma 1.83(a)]{BW82}, and \cite{Str1} for notation).
\end{remark}

Since the Lie algebra $(\cO,*)^-$ is simple, in particular there is no nonzero element in $\cO$ annihilated by $\Der(\cO,*)$. This is in contrast with the situation for automorphisms (see Proposition \ref{pr:e_fixed}).

\bigskip

\section{Automorphisms}

Equation \eqref{eq:x*y*x} shows that any automorphism $\varphi$ of the Okubo algebra $(\cO,*)$ belongs to the orthogonal group relative to its norm $n$.

If the characteristic of $\FF$ is not three, then $n(x,y)=-\tr(xy)$ for any $x,y\in \cO=\frsl_3(\FF)$. Then \eqref{eq:x*y} easily gives that, if we extend $\varphi$ to $M_3(\FF)$ by $\varphi(1)=1$, we obtain an automorphism of $M_3(\FF)$, and conversely.

\begin{theorem}
Let $\FF$ be a field of characteristic not three, then $\Aut(\cO,*)$ is isomorphic to the projective general linear group $\mathrm{PGL}_3(\FF)$. Actually,
all this is functorial, so this is valid for the affine group schemes: $\AAut(\cO,*)$ is isomorphic to $\mathbf{PGL}_3$.
\end{theorem}

In particular, there is no nonzero element fixed by all the automorphisms of $(\cO,*)$.
\smallskip

Assume for the rest of this section that the characteristic of $\FF$ is three. The situation is drastically different in this case.

As in the previous section, we may identify $(\cO,*,n)$ with $(\FF_0[x,y],*,n)$ with $x_{i,j}\leftrightarrow x^iy^j$, for $-1\leq i,j\leq 1$, $(i,j)\ne (0,0)$.

\begin{theorem}[\cite{Eld99}]
The restriction to $\FF_0[x,y]$ gives an isomorphism between $\Aut(\FF[x,y],\diamond)$ and $\Aut(\cO,*)$. Moreover, $\Aut(\cO,*)$  is the semidirect product of its unipotent radical, of dimension $5$, and a closed subgroup isomorphic to $\textrm{SL}_2(\FF)$.
\end{theorem}

Thus any $\varphi\in\Aut(\cO,*)$ extends to an automorphism, also denoted by $\varphi$, of $(\FF[x,y],\diamond)$ with $\varphi(1)=1$. But for any $f,g\in \FF[x,y]$,
\[
fg=\frac{1}{2}(f\diamond g+g\diamond f),
\]
(Proposition \ref{pr:Okubo_nodal}) so $\varphi$ is an automorphism too of the unital commutative and associative algebra of truncated polynomials $\FF[x,y]$. The nilpotent radical $\cN$ of this algebra is the subalgebra generated by $x-1$ and $y-1$, and $\cN^4=\FF(x-1)^2(y-1)^2$.
Note that
\[
(x-1)^2(y-1)^2=1+x+x^2+y+y^2+xy+x^2y^2+x^2y+xy^2=1+e,
\]
with
\begin{equation}\label{eq:e}
\begin{split}
e&=x+x^2+y+y^2+xy+x^2y^2+x^2y+xy^2\\
    &=x_{1,0}+x_{-1,0}+x_{0,1}+x_{0,-1}+x_{1,1}+x_{-1,-1}+x_{-1,1}+x_{1,-1}.
\end{split}
\end{equation}
(The sum of the basic elements in Table \ref{tb:Okubo}.) This element $e$ is an idempotent: $e*e=e$. Thus $\cN^4=\FF(1+e)$, and the automorphism $\varphi$ stabilizes $\cN^4$ so $\varphi(1+e)=\alpha(1+e)$ for some $0\ne\alpha\in \FF$. But $\varphi(1)=1$ and $\varphi$ leaves invariant $\FF_0[x,y]$, so $\alpha=1$ and $\varphi(e)=e$. In other words, this special idempotent $e$ is invariant under the action of $\Aut(\cO,*)$. Actually, straightforward computations give the following result:

\begin{proposition}\label{pr:e_fixed}
The subalgebra of the elements fixed by $\Aut(\cO,*)$ is precisely $\FF e$, where $e$ is the idempotent in \eqref{eq:e}.
\end{proposition}

\smallskip

Consider now the affine group scheme $\Gs=\AAut(\cO,*)$ and the stabilizer of the idempotent $e$ above: $\Hs=\AAut(\cO,*,e)$. Then $\Gs(\FF)=\Hs(\FF)=\Aut(\cO,*)$.

Recall \cite[\S 21]{KMRT} that given an algebraic group scheme $\Ks$ there is an associated smooth algebraic group scheme $\Ks_{\textrm{red}}$, the largest smooth closed subgroup of $\Ks$.

\begin{proposition}[{\cite[Corollary 10.8 and Proposition 10.10]{CEKT}}]
$\Hs$ is precisely $\Gs_{\textrm{red}}$.
\end{proposition}

However, $\Gs$ is not smooth, and hence the eight-dimensional subgroup $\Hs$ does not fill the whole $\Gs$. To see what is missing, remember that the assignment $\deg(x_{i,j})=(i,j)\pmod 3$ gives a grading of $\cO$ by $\ZZ_3^2$, and hence \cite[\S 1.4]{EKmon} there is a group scheme homomorphism
\[
\bmu_3\times\bmu_3\rightarrow \Gs,
\]
where $\bmu_3$ is the group scheme of $3^{\textrm{rd}}$ roots of unity, such that for any unital commutative associative $\FF$-algebra $R$ and any $\alpha,\beta\in R$ with $\alpha^3=\beta^3=1$, the image of $(\alpha,\beta)$ is the automorphism of $(\cO\otimes_\FF R,*)$ determined by
\[
x_{1,0}\otimes 1\mapsto x_{1,0}\otimes\alpha,\qquad x_{0,1}\otimes 1\mapsto x_{0,1}\otimes \beta.
\]
The image $\Ds$ under this homomorphism is the group scheme of diagonal automorphisms relative to the basis in Table \ref{tb:Okubo}.

\begin{theorem}[{\cite[Theorem 11.8 and Proposition 11.9]{CEKT}}]
For any unital commutative associative $\FF$-algebra $R$, $\Gs(R)=\Hs(R)\Ds(R)$ and $\Hs(R)\cap\Ds(R)=1$. Neither of the subgroups $\Hs$ and $\Ds$ are normal.
\end{theorem}

Note that $\bmu_3\times\bmu_3$ is not smooth ($\chr\FF=3$). Its Lie algebra has dimension $2$ and this accounts for the difference between the dimension of the Lie algebra $\Lie(\Gs)=\Der(\cO,*)$ (Theorem \ref{th:derivations}) and  the dimension of $\Aut(\cO,*)$.

\bigskip

\section{Idempotents}

Proposition \ref{pr:e_fixed} highlights a very special idempotent of the Okubo algebra in characteristic three.

Let $(\cO,*,n)$ be the Okubo algebra over $\FF$ (of any characteristic), and let $f$ be an idempotent: $f*f=f\ne 0$. Then $f=(f*f)*f=n(f)f$, so $n(f)=1$ (see Lemma \ref{le:symmetric}).

As in \cite{Petersson} or \cite{ElduquePerez}, consider the new multiplication on $\cO$ defined by
\begin{equation}\label{eq:Hurwitz_product}
x\cdot y=(f*x)*(y*f),
\end{equation}
for any $x,y\in\cO$. Then, for any $x,y\in\cO$,
\[
\begin{gathered}
f\cdot x=(f*f)*(x*f)=f*(x*f)=n(f)x=x,\\
x\cdot f=(f*x)*(f*f)=(f*x)*f=n(f)x=x,\\
n(x\cdot y)=n\bigl((f*x)*(y*f)\bigr)=n(f)^2n(x)n(y=n(x)n(y),
\end{gathered}
\]
so $(\cO,\cdot,n)$ is the Cayley algebra over $\FF$, i.e., the unital eight-dimensional composition algebra over $\FF$, with unity $1=f$.

For background on composition algebras the reader may consult \cite[Chapter VIII]{KMRT}.

\begin{remark}
If the characteristic of $\FF$ is three, and we take as our idempotent $f$ the idempotent $e$ in \eqref{eq:e} fixed by $\Aut(\cO,*)$, then it follows that $\Aut(\cO,*)$ is contained in the automorphism group $\Aut(\cO,\cdot)$ of the associated Cayley algebra, which is the simple group of type $G_2$.  This gives an algebraic explanation of the result by Tits on geometric triality in \cite[Theorem 10.1]{Tits}, as the group $\Aut(\cO,*)$ is the fixed subgroup of a triality automorphism of the spin group attached to the norm $n$.
\end{remark}

\smallskip

The first step in studying the idempotents of $(\cO,*)$ is to relate them to order three automorphisms (see \cite[(34.9)]{KMRT} and references therein):

\begin{proposition}\label{pr:tau}
The linear map
\begin{equation}\label{eq:tau}
\tau: x\mapsto f*(f*x)
\end{equation}
is an automorphism of order $3$ of both the Okubo algebra $(\cO,*,n)$ and the Cayley algebra $(\cO,\cdot,n)$.

Moreover, the subalgebra of elements fixed by $\tau$ coincides with the centralizer of $f$ in $(\cO,*)$: $\Centr_{(\cO,*)}(f)=\{x\in\cO: x*f=f*x\}$.
\end{proposition}
\begin{proof}
Because of equation \eqref{eq:xyx} we have
\[
\tau(x)=f*(f*x)=n(f,x)f-x*f,
\]
for any $x\in\cO$. Also, $\tau$ is bijective with inverse
\[
\tau^{-1}:x\mapsto (x*f)*f=n(f,x)f-f*x.
\]
For $x,y\in\cO$,
\[
\begin{split}
\tau(x)*\tau(y)&=\bigl(n(f,x)f-x*f\bigr)*\bigl(n(f,y)f-y*f\bigr)\\
    &=n(f,x)n(f,y)f-n(f,x)y-n(f,y)(x*f)*f+(x*f)*(y*f)\\
    &=n(f,y)\Bigl(n(f,x)f-(x*f)*f\Bigr)-n(f,x)y\\
    &\hspace{140pt} +\Bigl(n(x,y*f)f-((y*f)*f)*x\Bigr)\\
    &=n(f,y)f*x-n(f,x)y+n(x*y,f)f-n(f,y)f*x+(f*y)*x\\
    &=n(f,x*y)f-(x*y)*f\\
    &=\tau(x*y).
\end{split}
\]
Hence, $\tau\in\Aut(\cO,*)$. Since $\tau(f)=f$, it is clear that $\tau$ is in $\Aut(\cO,\cdot)$ too. Also,
\[
f*(f*(f*x))=n(f,f*x)f-(f*x)*f=n(f,x)f-x,
\]
which is minus the reflection of $x$ relative to the hyperplane orthogonal to $f$. Hence, the sixth power of the left multiplication by $f$ is the identity: $\tau^3=1$. On the other hand, an element $x$ is fixed by $\tau$ if and only if $x=f*(f*x)$. Multiply on the right on both sides by $f$ to get $x*f=f*x$, so the subalgebra of fixed elements by $\tau$ is the centralizer of $f$.  In particular, $\tau=1$ if and only if  $f$ is in the commutative center $\cK(\cO,*)$ in \eqref{eq:KO}, which is trivial. Hence, the order of $\tau$ is $3$.
\end{proof}

\begin{remark}\label{re:para-Hurwitz}
Under the conditions of Proposition \ref{pr:tau}, the multiplication in the Okubo algebra is recovered as follows:
\[
\begin{split}
x*y&=(x*f)\cdot(f*y)\\
  &=\tau(\bar x)\cdot\tau^2(\bar y),
\end{split}
\]
where $\bar x=n(f,x)f-x$. The composition algebras with this kind of multiplication are called Petersson algebras \cite{Petersson}.

Given a Hurwitz algebra (i.e.; a unital composition algebra) $(\cC,\cdot,n)$, the algebra defined on $\cC$ but with new product $x\bullet y=\bar x\cdot \bar y$, where $\bar x=n(1,x)1-x$, is called the associated \emph{para-Hurwitz algebra}. In this way we may talk about para-Cayley algebras, para-quaternion algebras, or para-quadratic algebras, depending on $(\cC,\cdot,n)$ being a Cayley algebra ($\dim\cC=8$), a quaternion algebra ($\dim\cC=4$) or a quadratic \'etale algebra ($\dim\cC=2$). The unity of the Hurwitz algebra belongs to the commutative center of the associated para-Hurwitz algebra.
\end{remark}

If the characteristic of $\FF$ is not three, the subalgebra of $(\cO,\cdot)$ of the elements fixed by $\tau$ is a composition subalgebra, and hence its dimension is $2$ or $4$. However, \cite[Theorem 3.5]{ElduquePerez} shows that the first case is impossible. In terms of \eqref{eq:x*y}, if an element $f\in\frsl_3(\FF)$ is an idempotent of the Okubo algebra, its minimal polynomial as an element in $M_3(\FF)$ has degree $\leq 2$. One deduces easily that, up to conjugation, the only possibility is
\[
f=\frac{1}{\omega-\omega^2}\begin{pmatrix} 2&0&0\\ 0&-1&0\\ 0&0&-1\end{pmatrix}.
\]
Note that in terms of Table \ref{tb:Okubo}, the elements $x_{i,j}+x_{-i,-j}$ are all idempotents.
We summarize the situation in the next result:

\begin{theorem}
All  idempotents in the Okubo algebra over $\FF$, $\chr\FF\ne 3$, are conjugate under the automorphism group.
\end{theorem}

\smallskip

Again, the situation is quite different in characteristic three. Assume from now on that the characteristic of $\FF$ is three.

Given an endomorphism $\varphi$ of a vector space $V$, denote by $V^\varphi$ the subspace of the elements fixed by $\varphi$. The \emph{rank} of a quadratic form $q:V\rightarrow\FF$ denotes the difference $\dim V-\dim V'$, where $V'\bydef\{v\in V: q(v)=q(v,V)=0\}$. (See \cite[I.2.1]{Che54}.)

\begin{lemma}\label{le:tau}
Let $(\cC,\cdot,n)$ be the Cayley algebra over $\FF$ ($\chr\FF=3$) and let $\tau$ be an automorphism of $(\cC,\cdot)$ of order $3$. Then we have one of the following possibilities:
\begin{romanenumerate}
\item $\dim \cC^\tau=6$. In this case there exists a quaternion subalgebra $\cQ$ of $\cC$ contained in $\cC^\tau$, and the rank of the restriction of the norm $n$ to $\cC^\tau$ is $4$ (so $\cC^\tau=\cQ\oplus (\cC^\tau)^\perp$).
\item $\dim\cC^\tau=4$ and the rank of the restriction of the norm $n$ to $\cC^\tau$ is $2$. In this case there exists a quadratic \'etale subalgebra $\cK$ (i.e., isomorphic to $\FF\times\FF$ since $\FF$ is algebraically closed) contained in $\cC^\tau$.
\item $\dim\cC^\tau=4$ and the rank of the restriction of the norm $n$ to $\cC^\tau$ is $1$. In this case $\cC^\tau=\FF 1\oplus \cW$, where $\cW$ is a totally isotropic three-dimensional space orthogonal to $1$.
\end{romanenumerate}
\end{lemma}
\begin{proof}
Obviously $1\in\cC^\tau$. Also, the restriction of $n$ to $\cC^\tau$ is not regular, as this would imply that $(\cC^\tau)^\perp\cap\cC^\tau=0$, but $(\cC^\tau)^\perp$ is $\tau$-invariant, so it contains nonzero eigenvectors for $\tau$. 

Let $\cC_0$ be the subspace orthogonal to $1$ relative to $n$. The restriction $\tau\vert_{\cC_0}$ satisfies $(\tau-1)^3=0\ne \tau -1$, which implies $3\leq \dim\cC_0^\tau\leq 6$. If $\cC_0^\tau$ is totally isotropic, we obtain the last possibility. Otherwise, there is an element $a\in\cC_0$ with $n(a)\ne 0$. Since $\FF$ is algebraically closed we may assume $n(a)=-1$, and hence the elements $e_1=\frac{1}{2}(1+a)$ and $e_2=\frac{1}{2}(1-a)$ are idempotents whose sum is $1$ contained in $\cC^\tau$. Consider the corresponding Peirce decomposition (see, for instance, \cite[\S 4.1]{EKmon}):
\[
\cC=\FF e_1\oplus\FF e_2\oplus \cU\oplus\cV,
\]
where $\cU=\{x\in\cC: e_1\cdot x=x=x\cdot e_2\}$ and $\cV=\{ x\in\cC: e_2\cdot x=x=x\cdot e_1\}$. Both $\cU$ and $\cV$ are invariant under $\tau$, and they are totally isotropic and paired by $n$. 

Moreover, $n$ is invariant under $\tau$, so $\dim \cU^\tau=\dim\cV^\tau$ is either $1$ or $2$. If this common dimension is $1$, then $n(\cU^\tau,\cV^\tau)=0$, as otherwise the restriction of $n$ to $\cC^\tau$ would be regular. Hence we obtain the second possibility. On the contrary, if this common dimension is $2$, then there are elements $u\in\cU^\tau$ and $v\in\cV^\tau$ such that $n(u,v)\ne 0$, the subspace $\cQ=\FF e_1\oplus\FF e_2\oplus\FF u\oplus\FF v$ is a quaternion subalgebra contained in $\cC^\tau$ and $\cC^\tau\cap\cQ^\perp$ is a two-dimensional totally isotropic subspace (if there were an element $x\in\cC^\tau\cap\cQ^\perp$ with $n(x)\ne 0$, then $\cC=\cQ\oplus\cQ x$ would be contained in $\cC^\tau$, so $\tau$ would be the identity). We thus obtain the first possibility in the Lemma.
\end{proof}

Recall that if $f$ is an idempotent of the Okubo algebra $(\cO,*,n)$, and we consider the Cayley algebra $(\cO,\cdot,n)$ with the product in \eqref{eq:Hurwitz_product}, and its order $3$ automorphism $\tau$ in Proposition \ref{pr:tau}, then $\cO^\tau=\Centr_{(\cO,*)}(f)$.

In this situation, given a quaternion subalgebra $(\cQ,\cdot,n)$ of $(\cO,\cdot,n)$ with $\cQ\subseteq \cO^\tau$, the multiplication $*$ in $\cQ$ is given by $x*y=\bar x\cdot \bar y$, because $\tau\vert_\cQ=1$, so $(\cQ,*,n)$ is a para-quaternion algebra (see Remark \ref{re:para-Hurwitz}).

Therefore, Lemma \ref{le:tau} shows that there are three different types of idempotents in the Okubo algebra:

\begin{df}\label{df:idempotents_3}
Let $f$ be an idempotent of the Okubo algebra $(\cO,*,n)$  ($\chr\FF=3$). Then $f$ is said to be:
\begin{itemize}
\item \emph{quaternionic}, if its centralizer contains a para-quaternion algebra,
\item \emph{quadratic}, if its centralizer contains a para-quadratic algebra and no para-quaternion subalgebra,
\item \emph{singular}, otherwise.
\end{itemize}
\end{df}

Using the multiplication in Table \ref{tb:Okubo}, we get the following examples of idempotents of each type:
\begin{itemize}
\item \emph{Quaternionic}: $\sum_{-1\leq i,j\leq 1,\, (i,j)\ne (0,0)}x_{i,j}$.
\item \emph{Quadratic:} any of the idempotents  $x_{i,j}+x_{-i,-j}$, $-1\leq i,j\leq 1$, $(i,j)\ne (0,0)$.
\item \emph{Singular:} $-x_{1,0}-x_{0,1}-x_{-1,-1}+x_{-1,1}+x_{1,-1}$.
\end{itemize}

The example of quaternionic idempotent above is the idempotent in  \eqref{eq:e}. Actually, a stronger result has been proved in \cite[Theorem 9.13]{CEKT}:

\begin{theorem}
The Okubo algebra $(\cO,*,n)$ contains a unique quaternionic idempotent: the idempotent $e$ in \eqref{eq:e}.
\end{theorem}

For quadratic idempotents, the arguments in the proof of Lemma \ref{le:tau} can be used to prove that, if $f$ is such an idempotent of $(\cO,*,n)$, and $(\cO,\cdot,n)$ is the Cayley algebra with multiplication in \eqref{eq:Hurwitz_product}, then there are dual bases $\{u_1,u_2,u_3\}$ in $\cU$ and $\{v_1,v_2,v_3\}$ in $\cV$ such that $\tau(u_i)=u_{i+1}$ and $\tau(v_i)=v_{i+1}$ (indices modulo $3$). Since the multiplications $\cdot$ and $*$ are completely determined by these bases being dual relative to $n$, we get the `quadratic part' of our last result. The `singular part' is left as an exercise.

\begin{proposition}
The quadratic idempotents form a unique orbit under the action of $\Aut(\cO,*,n)$. The same happens for the singular idempotents.
\end{proposition}

%

\bigskip

\bibliographystyle{amsalpha}

\begin{thebibliography}{KMRT98}

\bibitem[AF54]{AF} A.-A.~Albert and M.-S.~Frank, \emph{Simple Lie algebras of characteristic $p$}, Univ. e Politec. Torino Rend. Sem. Mat. \textbf{14} (1954-1955), 117--139.

\bibitem[Blo58]{Block58}
    R.-E.~Block, \emph{New simple Lie algebras of prime characteristic},
   Trans. Amer. Math. Soc. \textbf{89} (1958), 421--449.

\bibitem[BW82]{BW82}
     R.-E.~Block and R.-L.~Wilson, \emph{The simple Lie p-algebras of rank two}, Ann. of Math. (2) \textbf{115} (1982), no.~1, 93--168.

\bibitem[CEKT13]{CEKT}
    V.~Chernousov, A.~Elduque, M.~Knus and J.-P.~Tignol, \emph{Algebraic Groups of Type $D_4$, Triality, and Composition Algebras}, Documenta Math. \textbf{18} (2013), 413--468.
    
\bibitem[Che54]{Che54}
    C.-C.~Chevalley, \emph{The algebraic theory of spinors}.
    Columbia University Press, New York, 1954.

\bibitem[Eld97]{Eld97}
    A.~Elduque, \emph{Symmetric composition algebras}, J.~Algebra, \textbf{196} (1997) 283--300.

\bibitem[Eld99a]{Eld99}
    A.~Elduque, \emph{Okubo algebras in characteristic $3$ and their automorphisms}, Comm. Algebra \textbf{27} (1999), no.~6, 3009--3030.

\bibitem[Eld99b]{Twisted}
    A.~Elduque, \emph{Okubo algebras and twisted polynomials}, Recent progress in algebra (Taejon/Seoul, 1997), 101--109, Contemp. Math. \textbf{224}, Amer. Math. Soc., 1999.

\bibitem[EK13]{EKmon} A.~Elduque and M.~Kochetov,  \emph{Gradings on simple {L}ie algebras}. Mathematical Surveys and Monographs \textbf{189},  American Mathematical Society, Providence, RI, 2013.

\bibitem[EM91]{EM91} A.~Elduque and H.-C.~Myung,
    \emph{Flexible composition algebras and Okubo algebras}, Comm. Algebra \textbf{19} (1993), no.~4, 1197--1227.

\bibitem[EM93]{EM93} A.~Elduque and H.-C.~Myung,
    \emph{On flexible composition algebras}, Comm. Algebra \textbf{21}
   (1993), no.~7, 2481--2505.

\bibitem[EP96]{ElduquePerez}
    A.~Elduque and J.-M.~P\'erez-Izquierdo, \emph{Composition algebras with associative bilinear form}, Comm. Algebra, \textbf{24} (1996), no.~3, 1091--1116.

\bibitem[KMRT98]{KMRT} M.-A.~Knus, A.~Merkurjev, M.~Rost, and J.-P.~Tignol,  \emph{The book of involutions.} (With a preface in French by J. Tits.)
    American Mathematical Society Colloquium Publications \textbf{44}, American Mathematical Society, Providence, RI, 1998.

\bibitem[Kok58]{Kokoris58}
    L.-A. Kokoris, \emph{Simple modal noncommutative Jordan algebras},
   Proc. Amer. Math. Soc. \textbf{9} (1958), 652--654.

\bibitem[Oku78]{Oku78}S.~Okubo,
    \emph{Pseudo-quaternion and pseudo-octonion algebras},
   Hadronic J. \textbf{1} (1978), no.~4, 1250--1278.

\bibitem[Oku95]{Okubo_book}
    S.~Okubo, \emph{Introduction to octonion and other non-associative algebras in physics},
   Montroll Memorial Lecture Series in Mathematical Physics \textbf{2}, Cambridge University Press, Cambridge, 1995.

\bibitem[OO81]{OO81} S.~Okubo and J.-M. Osborn \emph{Algebras with nondegenerate associative symmetric bilinear forms
   permitting composition}, Comm. Algebra \textbf{9} (1981), no.~12,
    1233--1261.

\bibitem[Pet69]{Petersson}
    H.-P.~Petersson, \emph{Eine Identit\"at f\"unften Grades, der gewisse Isotope von Kompositions-Algebren gen\"ugen}, Math. Z. \textbf{109} (1969) 217--238.

\bibitem[Str72]{Str72}
    H.~Strade, \emph{Nodale nichtkommutative Jordanalgebren und Lie-Algebren bei Charakteristik $p>2$}, J.~Algebra \textbf{21} (1972), 353--377.

\bibitem[Str04]{Str1}
    H.~Strade, \emph{Simple Lie algebras over fields of positive characteristic. I}, de Gruyter Expositions in Mathematics \textbf{38}, Walter de Gruyter GmbH \& Co. KG, Berlin, 2004.

\bibitem[Str09]{Str2}
    H.~Strade, \emph{Simple Lie algebras over fields of positive characteristic. II\ (Classifying the absolute toral rank two case)}, de Gruyter Expositions in Mathematics \textbf{42}, Walter de Gruyter GmbH \& Co. KG, Berlin, 2009.

\bibitem[Str13]{Str3}
    H.~Strade, \emph{Simple Lie algebras over fields of positive characteristic. III\ (Completion of the classification)}, de Gruyter Expositions in Mathematics \textbf{57}, Walter de Gruyter GmbH \& Co. KG, Berlin, 2013.

\bibitem[Tit59]{Tits}
    J.~Tits, \emph{Sur la trialit\'e et certains groupes qui s'en d\'eduisent}, Publ. Math. Inst. Hautes \'{E}tud. Sci. \textbf{2} (1959), 13--60.

\end{thebibliography}

\end{document}